%% file: paper_adjoint_EPIRK.tex
\documentclass[]{gOMS2e}

\usepackage{epstopdf}
\usepackage{subfigure}
\usepackage{amsmath}

\usepackage{todonotes}
\usepackage{lmodern}
\usepackage{enumitem}
\usepackage{pgfplots}
\usepackage{pgfplotstable}
\usepackage{hyperref}
\pgfplotsset{width=7cm,compat=1.3}
\usepackage{transparent}

\usepackage{siunitx}
\graphicspath{ {./IMAGES/} }

\theoremstyle{plain}
\newtheorem{theorem}{Theorem}[section]

\newtheorem{lemma}[theorem]{Lemma}

\theoremstyle{definition}

\theoremstyle{remark}
\newtheorem{remark}{Remark}

\newcommand{\error}{\mathcal{E}}
\newcommand{\parfrac}[2]{\frac{\partial #1}{\partial #2}}

\newcommand{\ti}{{t_{0}}}
\newcommand{\tf}{{t_{F}}}
\newcommand{\Jn}{\mathbf{J}_{n}}
\newcommand{\Tn}{\mathbf{T}_{n}}

\newcommand{\goal}{\overline{\Phi}}
\newcommand{\goalfull}{\Phi}
\newcommand{\bs}[1]{\boldsymbol{#1}}
\newcommand{\vf}[1]{\mathbf{#1}}

\newcommand{\RK}{\mathbb{R}^{K}}
\newcommand{\RP}{\mathbb{R}^{P}}
\newcommand{\R}{\mathbb{R}}
\newcommand{\abs}[1]{\left\lvert #1 \right\rvert}%
\newcommand{\norm}[1]{\left\lVert #1 \right\rVert}%
\DeclareMathOperator{\curl}{\nabla \times}

\begin{document}



\title{Solving Parameter Estimation Problems with Discrete Adjoint Exponential Integrators}

\author{
\name{Ulrich R\"{o}mer\textsuperscript{a}$^{\ast}$\thanks{$^\ast$Corresponding author. Email: roemer@temf.tu-darmstadt.de}, Mahesh Narayanamurthi\textsuperscript{b}
and Adrian Sandu\textsuperscript{b}}
\affil{\textsuperscript{a}Technische Universit\"{a}t Darmstadt, Darmstadt, Germany;
\textsuperscript{b}Virginia Tech, Blacksburg, USA}
}

\input{logo.tex}

\maketitle

\begin{abstract}
The solution of inverse problems in a variational setting finds best estimates of the model parameters by minimizing a cost function that penalizes the mismatch between model outputs and observations. The gradients required by the numerical optimization process are computed using adjoint models. Exponential integrators are a promising family of time discretizations for evolutionary partial differential equations. In order to allow the use of these discretizations in the context of inverse problems adjoints of exponential integrators are required. 
This work derives the discrete adjoint formulae for a W-type exponential propagation iterative methods of Runge-Kutta type (EPIRK-W). These methods allow arbitrary approximations of the Jacobian while maintaining the overall accuracy of the forward integration. The use of Jacobian approximation matrices that do not depend on the model state avoids the complex calculation of Hessians in the discrete adjoint formulae, and allows efficient adjoint code generation via algorithmic differentiation. We use the discrete EPIRK-W adjoints to solve inverse problems with the Lorenz-96 model and a computational magnetics benchmark test. Numerical results validate our theoretical derivations.
\end{abstract}


\begin{classcode}
34H05;34K29;34K35
\end{classcode}

\tableofcontents

\section{Introduction}
Differential equations are widely used to model the dynamics of physical processes. Even if the form of the equations perfectly captures the physical effects under consideration, the predictive capability of the mathematical model depends on the availability of accurate parameter values and initial conditions. Data assimilation, roughly defined as the solution of inverse problems with models defined by differential equations, fuses information from model outputs and (noisy) physical measurements to produce better estimates of parameter values. 

In this work we consider inverse problems where the dynamics is modeled by  ordinary differential equations or time dependent partial differential equations, and where a maximum a posteriori (MAP) estimate of the parameter values is computed. Specifically, we consider four-dimensional variational (4D-Var) data assimilation problems where physical measurements are taken at different points in time, and the parameters are updated using all available measurements in the given time window and imposing the system dynamics as strong constraints \cite{kalnay2003atmospheric,Sandu_2003_KPPSEN1,Sandu_2003_KPPSEN2,Sandu_2014_adjoint-parareal,Sandu_2005_adjointAQM,Sandu_2006_ICCS-KPP,Sandu_2006_RKdadj,Sandu_2010_KPP22_TLM_ADJ,Sandu_2010_inverseDADJ,Sandu_2011_advectionCMAQ,Sandu_2014_FATODE}. The time dimension is discretized using a W-method \cite{Steihaug_1979} formulation of a class of exponential integrators that can achieve a high order of accuracy with low stage count and have been shown to be capable of dealing with many stiff problems \cite{Tokman_2006_EPI,Tokman_2011_EPIRK}. To maximize the posterior distribution gradient based optimization techniques are employed. The gradients required in the optimization process are given by the discrete adjoint approach, i.e., are obtained by algorithmic differentiation of the chosen exponential integration method.

Consider the following initial value problem in autonomous form:
\begin{equation}
    \label{eqn:initial_value_problem}
    \frac{d\vf{y}}{dt} = \vf{f}(\vf{y}(t),\bs{\theta}), \qquad  \vf{y}(\ti) = \vf{y}_{\rm ini}(\bs{\theta}), \qquad \ti \leq t \leq \tf,
\end{equation}
where $\vf{y}(t) \in \RK$ is the model state,  $\bs{\theta} \in \RP$ is the vector of model parameters, and the right-hand side function $\vf{f}: \RK \times \RP  \rightarrow \RK$ is assumed to be smooth and continuously differentiable.

A general inverse problem to estimate the uncertain parameters $\bs{\theta}$ is formulated as follows:
\begin{subequations}
\label{eqn:optimal_control_problem}%
\begin{align}
\label{eqn:optimal_control_problem_objective}
	\bs{\theta_\textsc{map}} = \arg\min_{\bs{\theta}}\quad &  \goalfull(\bs{\theta}) =  
	u(\vf{y}(\ti),\bs{\theta}) + \int_{\ti}^{\tf} q(\vf{y}(t),\bs{\theta}) \, dt + w(\vf{y}(\tf),\bs{\theta}) \\
\label{eqn:optimal_control_problem_constraint}
	\text{subject to:}\quad & \vf{y}' = \vf{f}(\vf{y}(t),\bs{\theta}), \qquad \vf{y}(\ti) = \vf{y}_{\rm ini}(\bs{\theta}), \qquad \ti \leq t \leq \tf, \end{align}\cite{Sandu_2011_advectionCMAQ}
\end{subequations}
where the goal function $\goalfull : \RP \rightarrow \R$ is expressed as a sum of an integral involving a nonlinear function of the model state and parameters, $q : \RK \times \RP \rightarrow \R$ and nonlinear functions of the initial and final states and the model parameters, $u,w : \RK \times \RP \rightarrow \R$.

Two different approaches exist to solve problem \eqref{eqn:optimal_control_problem}. In the \textit{first optimize then discretize} approach, the first-order-optimality conditions for the continuous problem \eqref{eqn:optimal_control_problem} are derived first,  and then a discretization scheme is applied to numerically solve the resulting equations \cite{Betts_2010_practical}. This approach affords the freedom to choose different discretization schemes for the forward and adjoint problems. In the \textit{first discretize and then optimize} approach, the original problem \eqref{eqn:optimal_control_problem} is first discretized using an appropriate time stepping method to handle the constraint \eqref{eqn:optimal_control_problem_constraint} and approximating the integral in \eqref{eqn:optimal_control_problem_objective} with a quadrature. The first-order-optimality conditions for the discrete optimization problem are then derived \cite{Betts_2010_practical, betts2005discretize,Sandu_2015_fdvar-aposteriori}.
In this paper we study the latter approach, as it yields the exact gradient of the discrete system \cite{Sandu_2006_RKdadj,Sandu_2007_LMMdadj} and can be partially automated with the help of algorithmic differentiation \cite{griewank2008evaluating, neidinger2010introduction}. Other reasons for not preferring the former approach may include the difficulty in obtaining the necessary conditions for problems that are non-trivial and the need to re-derive these for each new problem \cite{betts2004direct,Betts_2010_practical}, and the additional overhead involved in the derivation and implementation of the corresponding discrete system. 


The paper is structured as follows. In Section \ref{sec:formulation_estimation} the exponential integrator and the discrete parameter estimation problem are introduced. Section \ref{sec:first_order_optimality} derives the first-order-optimality equations and the discrete adjoint of the exponential integrator. Numerical experiments are carried out in Section \ref{sec:applications} for two different test problems. Conclusions are drawn in Section \ref{sec:conclusions}.

\section{Formulation of the Discrete Parameter Estimation Problem}
\label{sec:formulation_estimation}

In this section we describe the discrete parameter estimation problem based on a time discretization that uses a \textit{W}-formulation  \cite{Steihaug_1979} of a class of exponential methods.

\subsection{EPIRK-W Time Discretization Methods}
\label{sec:epirkw}

Exponential Propagation Iterative Methods of Runge-Kutta type (EPIRK) \cite{Tokman_2006_EPI,Tokman_2011_EPIRK} use a very general ansatz among exponential time-stepping schemes that allows the construction of methods with low number of stages and high-order. A general $s$-stage EPIRK method reads:
\begin{subequations}
    \label{eqn:epirk_formulation_original}
    \begin{align}
    \begin{split}
    \vf{Y}_{n,i} &= \vf{y}_{n}\, + a_{i,1}\,  \bs{\psi}_{i,1}(g_{i,1}\,h\Jn)\, h\vf{f}(\vf{y}_{n}) \\
     &+ \displaystyle \sum_{j = 2}^{i} a_{i,j}\, \bs{\psi}_{i,j}(g_{i,j}h\,\Jn)\, h\Delta^{(j-1)}\vf{r}(\vf{y}_{n}), \qquad i = 1 \hdots s-1,
    \end{split}   \label{eqn:epirk_formulation_original_internal_stage}
    \\
    \vf{y}_{n+1} &= \vf{y}_{n}\, + b_{1}\, \bs{\psi}_{s,1}(g_{s,1}\,h\Jn)\, h\vf{f}(\vf{y}_{n}) + \displaystyle\sum_{j = 2}^{s} b_{j}\, \bs{\psi}_{s,j}(g_{s,j}h\,\Jn)\, h\Delta^{(j-1)}\vf{r}(\vf{y}_{n}), 
    \label{eqn:epirk_formulation_original_external_stage}
    \end{align}
\end{subequations}
where $\vf{y}_n$ is the state at the current time $t_n$, $\Jn = \partial \vf{f}/\partial \vf{y}|_{t_n}$ is the Jacobian matrix of the right hand side function \eqref{eqn:initial_value_problem}, $\vf{Y}_{n,1} \dots \vf{Y}_{n,s-1} $ are the intermediate stages, and $\vf{y}_{n+1}$ is the next-step solution. Each $\bs{\psi}_{i,j}(\cdot)$ matrix function is a linear combination of matrix functions $\bs{\varphi}_{k}(\cdot)$:
\begin{equation}
\label{eqn:psi_ij}
\bs{\psi}_{i,j}(\vf{Z}) = \sum_{k = 1}^{s} p_{i,j,k} \, \bs{\varphi}_{k}(\vf{Z}),
\end{equation}
where the scalar analytic function $\varphi_{k}(z)$ is defined as:
\begin{eqnarray}
\label{eqn:phi_k}
\varphi_k(z) = \int_{0}^1 e^{z\,(1 - \theta)} \frac{\theta^{k-1}}{(k-1)!} \, d\theta
= \sum_{i=0}^\infty \frac{z^i}{(i+k)!}, \hspace{1cm} k = 1, 2, \hdots,
\end{eqnarray}
with the matrix function counterparts defined using a series expansion. The functions $\varphi_{k}(z)$ satisfy the following recurrence relation:
\begin{eqnarray}
\label{eqn:phi_k_recursive_formulation_and_phi_k_0}
\varphi_{0}(z) = e^z; \quad
\varphi_{k+1}(z) = \frac{\varphi_k(z) - 1/k!}{z}, ~~k = 1, 2, \hdots; \quad \varphi_k(0) = \frac{1}{k!}.
\end{eqnarray}
The major computational cost of EPIRK methods is the evaluation of matrix-function-times-vector products of the form $\bs{\psi}_{i,j}\big(h\, \gamma\, \Jn)\cdot \vf{v}$, which are linear combinations of products $\bs{\varphi}_{k}\big(h\, \gamma\, \Jn)\cdot\vf{v}$. For large systems Krylov-subspace methods are the preferred choice for the evaluation of these products \cite{Hochbruck_1997_exp}. It has been shown that for certain kinds of problems, Krylov-based approximation of matrix exponential products converge faster than the corresponding Krylov-based linear system solves \cite{Hochbruck_1997_exp}. Here we use a Krylov-subspace based algorithm to compute the matrix exponential-like vector products as explained in Section \ref{sec:ad}.

In \eqref{eqn:epirk_formulation_original} the $\bs{\psi}$ matrix functions multiply the right-hand side function evaluated at the current state, $\vf{f}(\vf{y}_n)$, and the forward difference vector  $\Delta^{(j-1)}\vf{r}(\vf{y}_n)$. The forward difference is defined recursively with the help of the remainder function $\vf{r}(\vf{y})$ as follows:
\begin{equation}
\begin{split}
\vf{r}(\vf{y}) &:= \vf{f}(\vf{y}) - \vf{f}(\vf{y}_n) - \Jn (\vf{y} - \vf{y}_n), \\
\Delta^{(0)}\vf{r}(\vf{y}) &:= \vf{r}(\vf{y}),\qquad
\vf{Y}_{n,0} := \vf{y}_n, \\
\Delta^{(j)}\vf{r}(\vf{Y}_{n,i}) &= \Delta^{(j-1)}\vf{r}(\vf{Y}_{n,i+1}) - \Delta^{(j-1)}\vf{r}(\vf{Y}_{n,i}),\quad j\ge 1, \\
\Delta^{(j)}\vf{r}(\vf{y}_{n}) &= \Delta^{(j)}\vf{r}(\vf{Y}_{n,0}) = \Delta^{(j-1)}\vf{r}(\vf{Y}_{n,1}) - \Delta^{(j-1)}\vf{r}(\vf{Y}_{n,0}),\quad j\ge 1.
\end{split}
\end{equation}
Although classical EPIRK methods perform well on stiff problems \cite{Tokman_2011_EPIRK}, using them in the \textit{first discretize and then optimize} approach will require the computation of the Hessian of the right-hand side function. In order to circumvent this, we resort to a W-method formulation of the EPIRK method \cite{Narayanamurthi_2017_EPIRKK}. W-methods first introduced in \cite{Steihaug_1979} for Rosenbrock-type methods admit the use of arbitrary approximations in place of the exact Jacobian, while maintaining the derived order of convergence. 

In \cite{Narayanamurthi_2017_EPIRKK} the authors have recently developed W-type EPIRK methods. An $s$-stage EPIRK-W method reads:
\begin{subequations}
    \label{eqn:epirk_w_formulation_original}
    \begin{align}
    \begin{split}
    \vf{Y}_{n,i} &= \vf{y}_{n}\, + a_{i,1}\,  \bs{\psi}_{i,1}(g_{i,1}\,h\Tn)\, h\vf{f}(\vf{y}_{n}) \\
    &+ \displaystyle \sum_{j = 2}^{i} a_{i,j}\, \bs{\psi}_{i,j}(g_{i,j}h\,\Tn)\, h\Delta^{(j-1)}\vf{r}(\vf{y}_{n}), \qquad i = 1 \hdots s-1,
    \end{split}\label{eqn:epirk_w_formulation_original_internal_stage}\\
    \vf{y}_{n+1} &= \vf{y}_{n}\, + b_{1}\, \bs{\psi}_{s,1}(g_{s,1}\,h\Tn)\, h\vf{f}(\vf{y}_{n}) + \displaystyle\sum_{j = 2}^{s} b_{j}\, \bs{\psi}_{s,j}(g_{s,j}h\,\Tn)\, h\Delta^{(j-1)}\vf{r}(\vf{y}_{n}),  \label{eqn:epirk_w_formulation_original_external_stage}
    \end{align}
\end{subequations}
where the exact Jacobian $\Jn$ in \eqref{eqn:epirk_formulation_original} has been replaced with an arbitrary approximation $\Tn$. The exact implication of this will become clear from the discussion that follows. To simplify notation, we define:
\[
\mathbf{A}_{n,i,j} := a_{i,j}\,  \bs{\psi}_{i,j}(g_{i,j}\,h\,\Tn), \qquad 
\mathbf{B}_{n,j} := b_{j}\, \bs{\psi}_{s,j}(g_{s,j}\,h\,\Tn). 
\]
With this notational change, one step of an EPIRK-W method reads:
\begin{subequations}
\label{eqn:epirkw_final_formulation}
\begin{align}
\vf{Y}_{n,i} &= \vf{y}_{n}\, + \mathbf{A}_{n,i,1}\, h\vf{f}(\vf{y}_{n}) + \displaystyle \sum_{j = 2}^{i} \mathbf{A}_{n,i,j}\, h\Delta^{(j-1)}\vf{r}(\vf{y}_{n}), \quad i = 1, \hdots, s - 1. \label{eqn:epirkw_final_formulation_internal_stages}\\
\vf{y}_{n+1} &= \vf{y}_{n}\, + \mathbf{B}_{n,1}\, h\vf{f}(\vf{y}_{n}) + \displaystyle\sum_{j = 2}^{s} \mathbf{B}_{n,j}\, h\Delta^{(j-1)}\vf{r}(\vf{y}_{n}). \label{eqn:epirkw_final_formulation_external_stage} 
\end{align}
\end{subequations}
%

\subsection{The Discrete Formulation of the Parameter Estimation Problem}
\label{sec:discrete-opt}

In the \textit{first discretize and then optimize} approach to solve \eqref{eqn:optimal_control_problem} we partition the time interval into $N$ subintervals with nodes $t_0, t_1, \hdots, t_N = \tf$. The constraint \texttt{ODE} \eqref{eqn:optimal_control_problem_constraint} is replaced by its EPIRK-W discretization \eqref{eqn:epirkw_final_formulation} on each subinterval $[t_n,t_{n+1}]$. The integral term of the goal function \eqref{eqn:optimal_control_problem_objective} is also discretized by a quadrature rule evaluated on these time nodes. The discrete version of the parameter estimation problem \eqref{eqn:optimal_control_problem} reads:
\begin{subequations}
\label{eqn:optimal_control_problem_using_epirkw}
\begin{align}
	\bs{\theta_\textsc{map}} =\arg\min_{\bs{\theta}}\quad &  \widetilde{\goalfull}(\bs{\theta}) = \sum_{k = 0}^{N} q_k(\vf{y}_{k},\bs{\theta}) \label{eqn:optimal_control_problem_using_epirkw_objective}\\
	\text{subject to:} \quad & 
	\vf{Y}_{n,i} = \vf{y}_{n}\, + \mathbf{A}_{n,i,1}\, h\,\vf{f}(\vf{y}_{n},\bs{\theta}) \nonumber  \\
	&\qquad + \displaystyle \sum_{j = 2}^{i} \mathbf{A}_{n,i,j} h\Delta^{(j-1)}\vf{r}(\vf{y}_{n},\bs{\theta}), \quad i = 1, \hdots, s - 1.  \label{eqn:epirkw_formulation_internal_stages}\\
	& \vf{y}_{n+1} = \vf{y}_{n}\, + \mathbf{B}_{n,1}\, h\,\vf{f}(\vf{y}_{n},\bs{\theta}) \nonumber \\
	& \qquad + \displaystyle\sum_{j = 2}^{s} \mathbf{B}_{n,j}\, h\Delta^{(j-1)}\vf{r}(\vf{y}_{n},\bs{\theta}), \quad  0\leq n \leq N-1, \label{eqn:epirkw_formulation_final_stages}  \\
	&\vf{y}_0 = \vf{y}_{\rm ini}(\bs{\theta}).\label{eqn:epirkw_formulation_initial} 
\end{align}
\end{subequations}
The definition  of the discrete cost function \eqref{eqn:optimal_control_problem_using_epirkw_objective} includes the initial term $u(\vf{y}_{\rm ini},\bs{\theta})$ in $q_0(\vf{y}_{\rm ini},\bs{\theta})$, and the final term 
$w(\vf{y}_{N},\bs{\theta})$ in $q_N(\vf{y}_{N},\bs{\theta})$. 

%

We proceed by describing a special case of the optimization problem \eqref{eqn:optimal_control_problem_using_epirkw}, that will also be used in the numerical examples of Section \ref{sec:applications}.

\subsection{4D-Var Data Assimilation}
\label{sec:fdvar}
The four-dimensional variational data assimilation problem is a special case of \eqref{eqn:optimal_control_problem_using_epirkw}.
Adopting the notation of \cite{singh2013practical}, let $\bs{\theta_\textsc{true}}$ denote the true but unknown parameter vector and $\bs{\theta}_\textsc{b}$ its prior, which is also referred to as background. We assume that $\error_\textsc{b} = \bs{\theta}_\textsc{b} - \bs{\theta_\textsc{true}}$ is normally distributed with zero mean and covariance matrix $\mathbf{B}$. Let measurements $\vf{z}_i$ be given at time points $t_{j_i},\, i=1,\ldots,N_\textsc{obs}$, which are assumed to be a subset of the nodes of the time interval, for simplicity. Let $\mathcal{H}$ denote an observation operator that maps a state $\vf{y}_{j_i}$, at time $t_{j_i}$, to the space of observations. The measurement error $\error^\mathrm{obs}_{i} = \mathcal{H}(\vf{y}_{\textsc{true},j_i}) - \vf{z}_i$, where $\vf{y}_{\textsc{true},j_i}$ represents the state associated to $\bs{\theta_\textsc{true}}$ at time $t_{j_i}$.We assume a normal distribution for $\error^\mathrm{obs}_{i}$, with mean zero and covariance matrix $\mathbf{R}_i$. Assuming no model errors, the 4D-Var cost function and gradient are defined as follows:
\begin{align}
\widetilde{\goalfull}(\bs{\theta}) &= \frac{1}{2} \left (\bs{\theta} - \bs{\theta}_\textsc{b} \right)^T \mathbf{B}^{-1} \left (\bs{\theta} - \bs{\theta}_\textsc{b} \right) + \frac{1}{2} \sum_{i=1}^{N_\textsc{obs}} \left(\mathcal{H}(\vf{y}_{j_i}) - \vf{z}_{i} \right)^T \mathbf{R}_{i}^{-1} \left(\mathcal{H}(\vf{y}_{j_i}) - \vf{z}_{i} \right),\label{eqn:4dvar_cost_function}\\
\nabla_{\bs{\theta}} \widetilde{\goalfull}(\bs{\theta}) &=	\mathbf{B}^{-1} \left(\bs{\theta} - \bs{\theta}^\text{B} \right) + \sum_{i=1}^{N_\textsc{obs}} \left(\frac{d \vf{y}_{j_i}}{d \bs{\theta}}\right)^T \left(\frac{d \mathcal{H}}{d\vf{y}}(\vf{y}_{j_i}) \right)^T \mathbf{R}_{i}^{-1} \left(\mathcal{H}(\vf{y}_{j_i}) - \vf{z}_{i} \right).\label{eqn:4dvar_gradient}
\end{align}
The value $\bs{\theta_\textsc{map}}$ minimizing \eqref{eqn:optimal_control_problem_using_epirkw} with the goal function given by \eqref{eqn:4dvar_cost_function}, represents the maximum-likelihood estimate of $\bs{\theta_\textsc{true}}$. In a data assimilation context, $\bs{\theta_\textsc{map}}$ is also referred to as analysis. 

\begin{remark}
In parameter estimation problem \eqref{eqn:optimal_control_problem}, the goal function and \texttt{ODE} constraint depend on a vector of unknown parameters $\bs{\theta} \in \mathbb{R}^P$ that need to be inferred from measurements. In the general case, computation of gradient of the cost function, as shown in \eqref{eqn:gradient_goal_function}, involves complicated derivatives with respect to $\bs{\theta}$. This can be circumvented by converting this into an initial state estimation problem. To this end, the \texttt{ODE} system can be extended with the addition of parameters $\bs{\tau}(t) = \bs{\theta}$ as formal variables to the state vector as follows:
\begin{equation}
{\tilde{\vf{y}}}' = 
\begin{bmatrix}
\vf{y} \\
\bs{\tau}
\end{bmatrix}'
= 
\begin{bmatrix}
\vf{f}(\vf{y},\bs{\tau}) \\
\bs{0}
\end{bmatrix}
= \tilde{\vf{f}}(\tilde{\vf{y}})
,\quad 
\tilde{\vf{y}}(t_0)=
\begin{bmatrix}
\vf{y}(t_0) \\
\bs{\tau}(t_0)
\end{bmatrix}
= 
\begin{bmatrix}
\vf{y}_{\rm ini} \\
\bs{\theta}
\end{bmatrix}
= \tilde{\vf{y}}_{\rm ini}.
\end{equation}
\end{remark}

\section{First Order Optimality Conditions and the Discrete EPIRK-W Adjoint}
\label{sec:first_order_optimality}
This section is devoted to the first-order-optimality system associated with the discrete inverse problem \eqref{eqn:optimal_control_problem_using_epirkw}, which involves discrete adjoints for the EPIRK-W method. Various authors have addressed discrete adjoints in optimal control, constrained by ordinary differential equations. In \cite{hager2000runge} discrete adjoints for Runge-Kutta methods were derived together with control specific order conditions. A similar approach was taken in \cite{lang2013w} for Runge-Kutta W-methods. In both works, the discrete adjoint method could be reformulated again as a Runge-Kutta method and this was exploited in the subsequent error analysis. A general discussion on the use of discrete time integration adjoints in the solution of inverse problems is presented in \cite{Sandu_2010_inverseDADJ}. Discrete adjoints of Rosenbrock methods are formulated in \cite{Sandu_2000_RosAdjoint_book,Sandu_2000_RosAdjoint}. Theoretical properties of general discrete Runge-Kutta adjoints are proved in \cite{Sandu_2006_RKdadj}, and theoretical properties of discrete adjoints of linear multistep methods in \cite{Sandu_2007_LMMdadj}. Discrete adjoints of variable-step integrators are studied in \cite{Sandu_2009_dadjTimestep}. Efficient implementations of Runge-Kutta adjoints are provided by software developed by the authors \cite{Sandu_2006_ICCS-KPP,Sandu_2009_DENSERKS,Sandu_2009_HPC-GeosChem,Sandu_2010_KPP22_TLM_ADJ,Sandu_2014_FATODE,Sandu_2015_MATLODE}.

This work is concerned  with deriving the discrete adjoints of EPIRK-W methods. As we will see, the discrete adjoint equations do not readily show the structure of an exponential method.  Discrete adjoints for another closely related class of exponential methods can be found in \cite{rothauge2016discrete}.

\subsection{The Discrete Lagrangian}
\label{sec:dlag}

The first-order-optimality system is obtained by seeking a stationary point of the Lagrangian of the discrete optimal control problem \eqref{eqn:optimal_control_problem_using_epirkw}. We introduce the compact notation:
\begin{equation}
\vf{Y}_{n} = \left(\vf{Y}_{n,1}^T, \ldots, \vf{Y}_{n,s-1}^T  \right )^T, \quad \hat{\vf{Y}} = \left(\vf{Y}_{0}^T, \ldots, \vf{Y}_{N-1}^T  \right )^T, \quad \hat{\vf{y}} = \left(\vf{y}_{\rm ini}^T, \ldots, \vf{y}_{N}^T  \right )^T.
\end{equation}
Let $\bs{\Lambda}_{n,i}$, $\bs{\lambda}_n$ denote the Lagrange multipliers associated to  \eqref{eqn:epirkw_formulation_internal_stages} and \eqref{eqn:epirkw_formulation_final_stages}, respectively. The Lagrangian of the discrete optimal control problem \eqref{eqn:optimal_control_problem_using_epirkw} reads:
%
\[
\mathcal{L}(\hat{\vf{y}},\hat{\vf{Y}},\hat{\bs{\lambda}},\hat{\bs{\Lambda}},\bs{\theta}) =  \sum_{k=0}^N q_k\big(\vf{y}_k,\bs{\theta}\big) - \mathcal{L}_{\textnormal{ext}}\big(\hat{\vf{y}},\hat{\vf{Y}},\hat{\bs{\lambda}},\bs{\theta}\big) - \mathcal{L}_{\rm int}\big(\hat{\vf{y}},\hat{\vf{Y}},\hat{\bs{\Lambda}},\bs{\theta}\big),
\]
where:
\begin{subequations}
\begin{align}
\label{eqn:external_Lagrangian}
\mathcal{L}_{\textnormal{ext}} &= \bs{\lambda}_0^T\big(\vf{y}_0 - \vf{y}_{\rm ini}(\bs{\theta})\big) \notag \\
&+ \sum_{k = 0}^{N - 1} \bs{\lambda}_{k+1}^{T} \left(\vf{y}_{k+1} - \vf{y}_{k}\, - \mathbf{B}_{k,1}\, h \vf{f}(\vf{y}_{k},\bs{\theta}) - \displaystyle\sum_{j = 2}^{s} \mathbf{B}_{k,j}\, h\Delta^{(j-1)}\vf{r}(\vf{y}_{k},\bs{\theta})\right),\\
\label{eqn:internal_Lagrangian}
\mathcal{L}_{\rm int} &= \sum_{k = 0}^{N - 1}  \sum_{m = 1}^{s - 1} \bs{\Lambda}_{k,m}^{T} \left(\vf{Y}_{k,m} - \vf{y}_{k}\, - \mathbf{A}_{k,m,1}\, h \vf{f}(\vf{y}_{k},\bs{\theta}) - \displaystyle \sum_{j = 2}^{m} \mathbf{A}_{k,m,j}\, h\Delta^{(j-1)}\vf{r}(\vf{y}_{k},\bs{\theta})\right).
\end{align}
\end{subequations}
In the above equations, the forward difference operator $\Delta^{(j-1)}\vf{r}(\vf{y}_{k},\bs{\theta})$ can be written in a closed form expression as
\begin{equation}
\Delta^{(j-1)}\vf{r}(\vf{y}_{k},\bs{\theta}) = \sum_{\ell = 0}^{j -1} C_{\ell,j}\, \vf{r}(\vf{Y}_{k, j-1-\ell},\bs{\theta}), 
\quad C_{\ell,j}= (-1)^{\ell}\, \binom{j-1}{\ell},
\nonumber
\end{equation}
%
where $\vf{r}(\vf{Y}_{k,j-1-\ell},\bs{\theta})$ is the remainder term of the first-order Taylor expansion of $\vf{f}$ around $\vf{y}_k$, evaluated at $\vf{Y}_{k,j-1-\ell}$ with the exact Jacobian $\vf{J}_{k}$ replaced by the approximation $\vf{T}_{k}$:
\begin{eqnarray}
\vf{r}(\vf{Y}_{k,j-1-\ell},\bs{\theta}) = \vf{f}(\vf{Y}_{k,j-1-\ell},\bs{\theta}) - \vf{f}(\vf{y}_k,\bs{\theta}) - \mathbf{T}_k \cdot (\vf{Y}_{k,j-1-\ell} - \vf{y}_k) \label{eqn:remainder_approximation}.
\end{eqnarray}
It should be pointed out that in the above expressions $\vf{Y}_{k, 0} = \vf{y}_k$ was introduced for brevity in later expressions. 

Substituting \eqref{eqn:remainder_approximation} into \eqref{eqn:external_Lagrangian} and \eqref{eqn:internal_Lagrangian}, respectively, we obtain:
\begin{subequations}
\begin{align}
\mathcal{L}_{\textnormal{ext}}& = \bs{\lambda}_0^T\big(\vf{y}_0 - \vf{y}_{\rm ini}(\bs{\theta})\big) + \sum_{k = 0}^{N - 1} \bs{\lambda}_{k+1}^{T} \Bigg(\vf{y}_{k+1} - \vf{y}_{k}\, - \mathbf{B}_{k,1}\, h \vf{f}(\vf{y}_{k},\bs{\theta}) \nonumber\\
&- \displaystyle\sum_{j = 2}^{s} \mathbf{B}_{k,j}\, h\bigg(\sum_{\ell = 0}^{j -1} C_{\ell,j} \big(\vf{f}(\vf{Y}_{k,j-1-\ell},\bs{\theta}) - \vf{f}(\vf{y}_k,\bs{\theta}) - \mathbf{T}_k (\vf{Y}_{k,j-1-\ell} - \vf{y}_k)\big)\bigg)\Bigg),\nonumber\\
\mathcal{L}_{\rm int}&= \sum_{k = 0}^{N - 1}  \sum_{m = 1}^{s - 1} \bs{\Lambda}_{k,m}^{T} \Bigg(\vf{Y}_{k,m} - \vf{y}_{k}\, - \mathbf{A}_{k,m,1}\, h\vf{f}(\vf{y}_{k},\bs{\theta}) \nonumber\\
& - \displaystyle \sum_{j = 2}^{m} \mathbf{A}_{k,m,j}\, h\bigg(\sum_{\ell = 0}^{j -1} C_{\ell,j} \big(\vf{f}(\vf{Y}_{k,j-1-\ell},\bs{\theta}) - \vf{f}(\vf{y}_k,\bs{\theta}) - \mathbf{T}_k (\vf{Y}_{k,j-1-\ell} - \vf{y}_k)\big)\bigg)\Bigg).\nonumber
\end{align}
\end{subequations}
%

\subsection{Derivation of the Discrete Adjoint Equation}
\label{sec:discrete_adjoint}
The adjoint equations \eqref{eqn:discrete_adjoint_equations} are obtained by setting $\partial \mathcal{L}/\partial \vf{y}_n = 0$ and $\partial \mathcal{L}/\partial \vf{Y}_{n,i} = 0$  for all $n>0$. For a better readability, we compute various derivatives separately. We thereby adopt the convention that all terms for which subscripts turn out to be negative are automatically dropped or set to zero.

\subsubsection{Computing $\partial \mathcal{L}_{\textnormal{ext}}/\partial \vf{y}_n$}
\label{sec:dL-a}

For $0<n < N$, only the terms $k=n-1$ and $k=n$ in the sum $\sum_{k=0}^{N-1}$ give non-zero contributions, as follows:
\begin{align*}
&\frac{\partial \mathcal{L}_{\textnormal{ext}}}{\partial \vf{y}_n}  =  \bs{\lambda}_{n}^{T}\cdot \frac{\partial}{\partial \vf{y}_n} \Bigg\{ \vf{y}_{n} - \vf{y}_{n-1} - \mathbf{B}_{n-1,1}\, h\vf{f}(\vf{y}_{n-1},\bs{\theta})   \nonumber\\
& \quad - \sum_{j = 2}^{s} \mathbf{B}_{n-1,j}\; h\,\sum_{\ell = 0}^{j -1} C_{\ell,j} \bigg[ \vf{f}(\vf{Y}_{n-1,j-1-\ell},\bs{\theta}) - \vf{f}(\vf{y}_{n-1},\bs{\theta}) - \mathbf{T}_{n-1} (\vf{Y}_{n-1,j-1-\ell} - \vf{y}_{n-1})\bigg]\Bigg\}\nonumber\\
&  \quad +   \bs{\lambda}_{n+1}^{T}\cdot \frac{\partial}{\partial \vf{y}_n}\Bigg\{ \vf{y}_{n+1} - \vf{y}_{n} - \mathbf{B}_{n,1}h\vf{f}(\vf{y}_{n},\bs{\theta}) \nonumber \\
& \quad - \sum_{j = 2}^{s} \mathbf{B}_{n,j} \; h \sum_{\ell = 0}^{j -1} C_{\ell,j} \bigg[ \vf{f}(\vf{Y}_{n,j-1-\ell},\bs{\theta}) - \vf{f}(\vf{y}_{n},\bs{\theta}) - \mathbf{T}_{n} (\vf{Y}_{n,j-1-\ell} - \vf{y}_{n})\bigg]\Bigg\}\nonumber\\
&= \bs{\lambda}_{n}^{T}\big\{\mathbf{I} - 0\, - 0 - 0\big\} + \bs{\lambda}_{n+1}^{T}\cdot \Bigg\{ 0 - \mathbf{I}\, - \mathbf{B}_{n,1}\, h\mathbf{J}(\vf{y}_{n},\bs{\theta}) \nonumber\\
& \quad -\sum_{j = 2}^{s} \mathbf{B}_{n,j}\, h\sum_{\ell = 0}^{j -2} C_{\ell,j} \big[0 - \mathbf{J}(\vf{y}_{n},\bs{\theta}) - \mathbf{T}_{n} (0 - \mathbf{I})\big] \Bigg\}\nonumber\\
&=  \bs{\lambda}_{n}^{T} + \bs{\lambda}_{n+1}^{T}\cdot \bigg\{- \mathbf{I}\, - \mathbf{B}_{n,1}\, h\mathbf{J}(\vf{y}_{n},\bs{\theta})  -\sum_{j = 2}^{s} \mathbf{B}_{n,j}\, h\sum_{\ell = 0}^{j -2} C_{\ell,j} \big[\mathbf{T}_{n} - \mathbf{J}(\vf{y}_{n},\bs{\theta})\big]\bigg\}, \nonumber
\end{align*}
where $\mathbf{I}$ and $\mathbf{J}_n$ are the $\mathbb{R}^K \times \mathbb{R}^K$ identity matrix and the Jacobian matrix evaluated at $\vf{y}_n$, respectively. Also the definition $\vf{Y}_{k,0}=\vf{y}_k$ was used in the derivation above. 

For $n = N$, only the term with index $k=N-1$ remains and we obtain:
\begin{align*}
&\frac{\partial \mathcal{L}_{\textnormal{ext}}}{\partial \vf{y}_{N}} = \bs{\lambda}_{N}^{T}\cdot \frac{\partial}{\partial \vf{y}_{N}} \bigg\{ \vf{y}_{N} - \vf{y}_{N-1}\, - \mathbf{B}_{N-1,1}\, h\vf{f}(\vf{y}_{N-1},\bs{\theta}) \nonumber  \\
&\quad - \sum_{j = 2}^{s} \mathbf{B}_{N-1,j}\,  h\sum_{\ell = 0}^{j -1} C_{\ell,j} \big[\vf{f}(\vf{Y}_{N-1,j-1-\ell},\bs{\theta}) - \vf{f}(\vf{y}_{N-1},\bs{\theta})  - \mathbf{T}_{N-1} (\vf{Y}_{N-1,j-1-\ell} - \vf{y}_{N-1})\big]\bigg\} \\
& = \bs{\lambda}_{N}^{T}.
\end{align*}
For $n=0$ we compute:
\begin{align*}
\frac{\partial \mathcal{L}_{\textnormal{ext}}}{\partial \vf{y}_0} 
&= \bs{\lambda}_0^T +  \bs{\lambda}_{1}^{T}\cdot \frac{\partial}{\partial \vf{y}_0} \Bigg\{ \vf{y}_{1} - \vf{y}_0\, - \mathbf{B}_{0,1}\, h \vf{f}(\vf{y}_0,\bs{\theta})   \\
& \ - \sum_{j = 2}^{s} \mathbf{B}_{0,j}\,  h \sum_{\ell = 0}^{j -1} C_{\ell,j} \big[\vf{f}(\vf{Y}_{0,j-1-\ell},\bs{\theta}) - \vf{f}(\vf{y}_0,\bs{\theta}) - \mathbf{T}_{0} (\vf{Y}_{0,j-1-\ell} - \vf{y}_0)\big]\Bigg\} \\
&= \bs{\lambda}_0^T + \bs{\lambda}_{1}^{T}\cdot \bigg\{- \mathbf{I} - \mathbf{B}_{0,1}\, h \mathbf{J}(\vf{y}_{0},\bs{\theta}) - \sum_{j = 2}^{s} \mathbf{B}_{0,j}\,  h \sum_{\ell = 0}^{j - 2} C_{\ell,j} \big(\mathbf{T}_{0}- \mathbf{J}(\vf{y}_{0},\bs{\theta}) \big)\bigg\}. 
\end{align*}

\subsubsection{Computing $\partial \mathcal{L}_{\textnormal{ext}}/\partial \vf{Y}_{n,i}$}
\label{sec:dL-b}

The same nonzero terms in the sum remain when we differentiate with respect to $\vf{Y}_{n,i}$ as when we differentiate with respect to $\vf{y}_n$. For $0 \leq n < N$ we obtain:
\begin{align*}
\frac{\partial \mathcal{L}_{\textnormal{ext}}}{\partial \vf{Y}_{n,i}} &= -\bs{\lambda}_{n+1}^{T} \cdot \Bigg\{ \sum_{j = 2}^{s} \mathbf{B}_{n,j} \nonumber\\
& \cdot \frac{\partial}{\partial \vf{Y}_{n,i}}\bigg(h \sum_{\ell = 0}^{j -1} C_{\ell,j} \big[\vf{f}(\vf{Y}_{n,j-1-\ell},\bs{\theta}) - \vf{f}(\vf{y}_{n},\bs{\theta}) - \mathbf{T}_{n} (\vf{Y}_{n,j-1-\ell} - \vf{y}_{n})\big]\bigg)\Bigg\} \\
&= -\bs{\lambda}_{n+1}^{T} \cdot \bigg\{\sum_{j = 2}^{s} \mathbf{B}_{n,j}\, h\,C_{j-i-1,j}\big[\mathbf{J}(\vf{Y}_{n,i},\bs{\theta}) - \mathbf{T}_{n} \big]\bigg\}. \nonumber
\end{align*}
For $n=N$ the derivative is zero.

\subsubsection{Computing $\partial \mathcal{L}_{\rm int}/\partial \vf{y}_n$}
\label{sec:dL-c}

As parts of the derivation were already given in Sec. \ref{sec:dL-a} we omit intermediate steps. For $0 \leq n < N$ we obtain:
\begin{align*}
\frac{\partial \mathcal{L}_{\rm int}}{\partial \vf{y}_n} &=\sum_{m = 1}^{s - 1} \bs{\Lambda}_{n,m}^{T} \cdot  \frac{\partial}{\partial \vf{y}_n} \Bigg\{\vf{Y}_{n,m} - \vf{y}_{n}\, - \mathbf{A}_{n,m,1}\, \vf{f}(\vf{y}_{n},\bs{\theta}) \nonumber \\
& \quad -\sum_{j = 2}^{m} \mathbf{A}_{n,m,j}\, h\sum_{\ell = 0}^{j -1}C_{\ell,j} \bigg[\vf{f}(\vf{Y}_{n,j-1-\ell},\bs{\theta}) - \vf{f}(\vf{y}_n,\bs{\theta}) - \mathbf{T}_n (\vf{Y}_{n,j-1-\ell} - \vf{y}_n)\bigg]\Bigg\} \\
\,
&= \sum_{m = 1}^{s - 1} \bs{\Lambda}_{n,m}^{T} \cdot \Bigg\{- \mathbf{I}\, - \mathbf{A}_{n,m,1}\, h \mathbf{J}(\vf{y}_{n},\bs{\theta})  -\sum_{j = 2}^{m} \mathbf{A}_{n,m,j}\, h\sum_{\ell = 0}^{j -1} C_{\ell,j}\big[- \mathbf{J}(\vf{y}_n,\bs{\theta}) + \mathbf{T}_n\big]\Bigg\}, 
\end{align*}
whereas there is no contribution for $n=N$.

\subsubsection{Computing $\partial \mathcal{L}_{\rm int}/\partial \vf{Y}_{n,i}$}
\label{sec:dL-d}

Again, omitting several intermediate results due to the similarity to Sec. \ref{sec:dL-b} we obtain for $0\leq n<N$:
\begin{align*}
\frac{\partial \mathcal{L}_{\rm int}}{\partial \vf{Y}_{n,i}}& = \sum_{m = 1}^{s - 1} \bs{\Lambda}_{n,m}^{T} \cdot  \frac{\partial}{\partial \vf{Y}_{n,i}} \Bigg\{ \vf{Y}_{n,m} - \vf{y}_{n}\, - \mathbf{A}_{n,m,1}\, h \vf{f}(\vf{y}_{n},\bs{\theta}) \nonumber  \\
& \quad - \sum_{j = 2}^{m} \mathbf{A}_{n,m,j}\,  h\sum_{\ell = 0}^{j -1} C_{\ell,j} \Big[\vf{f}(\vf{Y}_{n,j-1-\ell},\bs{\theta}) - \vf{f}(\vf{y}_n,\bs{\theta}) - \mathbf{T}_n (\vf{Y}_{n,j-1-\ell} - \vf{y}_n)\Big]\Bigg\} \\
&= \bs{\Lambda}_{n,i}^{T} - \sum_{m = 1}^{s - 1} \bs{\Lambda}_{n,m}^{T} \cdot \bigg\{ \sum_{j = 2}^{m} \mathbf{A}_{n,m,j}\, h C_{j-i-1,j} \big[\mathbf{J}(\vf{Y}_{n,i},\bs{\theta}) - \mathbf{T}_n\big]\bigg\}.
\end{align*}

Finally, we observe that at the stationary point $(\hat{\vf{y}}^*,\hat{\vf{Y}}^*,\hat{\bs{\lambda}}^*,\hat{\bs{\Lambda}}^*)$ of the Lagrangian, both $\mathcal{L}_\textnormal{ext}$ and $\mathcal{L}_\textrm{int}$ vanish and hence
\begin{equation}
\goalfull(\bs{\theta}) = \mathcal{L}(\hat{\vf{y}}^*,\hat{\vf{Y}}^*,\hat{\bs{\lambda}}^*,\hat{\bs{\Lambda}}^*,\bs{\theta}).
\end{equation}
Differentiating this expression with respect to the parameter $\bs{\theta}$ yields
\begin{align} 
\frac{\partial \goalfull}{\partial \bs{\theta}} &=  \sum_{k=1}^N \frac{\partial q_k(\vf{y}_k,\bs{\theta})}{\partial \bs{\theta}} - \frac{\partial}{\partial \bs{\theta}} \mathcal{L}_{\textnormal{ext}}(\hat{\vf{y}}^*,\hat{\vf{Y}}^*,\hat{\bs{\lambda}}^*,\bs{\theta}) - \frac{\partial}{\partial \bs{\theta}} \mathcal{L}_{\rm int}(\hat{\vf{y}}^*,\hat{\vf{Y}}^*,\hat{\bs{\Lambda}}^*,\bs{\theta})
\end{align}
from which we obtain the expression for the gradient, given in (C) below, by inserting the respective expressions for $\mathcal{L}_{\textnormal{ext}}$ and $\mathcal{L}_{\rm int}$.

\subsection{First Order Optimality System}
\label{sec:fos}
The first-order-optimality conditions of \eqref{eqn:optimal_control_problem_using_epirkw} are obtained from the results of Sec. \ref{sec:dL-a}--\ref{sec:dL-d} by enforcing the derivatives of $\mathcal{L}$ to zero:
\begin{enumerate}[label=(\Alph*)]
\item The original EPIRK-W method, i.e., \eqref{eqn:epirkw_formulation_internal_stages}, \eqref{eqn:epirkw_formulation_final_stages} and \eqref{eqn:epirkw_formulation_initial}, is obtained by equating $\partial \mathcal{L}/\partial \bs{\lambda}_{n}=0$ and $\partial \mathcal{L}/\partial \bs{\Lambda}_{n,i}=0$, respectively.
\item  The discrete adjoint system is obtained as follows.
\begin{subequations}
\label{eqn:discrete_adjoint_equations}
Enforcing $\partial \mathcal{L}/\partial \vf{y}_{n} = 0$ for $n=N$ and $n=0,\ldots,N-1$ leads to:
\begin{equation}
\label{eqn:discrete_adjoint_equation_final_stages}
\begin{split}
\bs{\lambda}_N^T & = \frac{\partial q_N(\vf{y}_N,\bs{\theta})}{\partial \vf{y}_N}, \\
\bs{\lambda}_{n}^T &= \frac{\partial q_n(\vf{y}_n,\bs{\theta})}{\partial \vf{y}_n} + \sum_{m = 1}^{s - 1} \bs{\Lambda}_{n,m}^{T} \cdot  \\
&\cdot \bigg(\mathbf{I}\, + \mathbf{A}_{n,m,1}\, h \mathbf{J}(\vf{y}_{n},\bs{\theta})+\sum_{j = 2}^{m} \mathbf{A}_{n,m,j} \, h \sum_{\ell = 0}^{j -1} C_{\ell,j}\big[ \mathbf{T}_n - \mathbf{J}(\vf{y}_n,\bs{\theta}) \big]  \bigg).    
\end{split}
\end{equation}
Imposing $\partial \mathcal{L}/\partial \vf{Y}_{n,i} = 0$, for $n=1,\ldots,N-1$, $i=1,\ldots,s-1$, yields:
\begin{equation}
\label{eqn:discrete_adjoint_equation_internal_stages}
\begin{split}
\bs{\Lambda}_{n,i}^{T} &= \bs{\lambda}_{n+1}^{T} \cdot \bigg\{\sum_{j = 2}^{s} \mathbf{B}_{n,j}\, h\,C_{j-i-1,j}\big[\mathbf{J}(\vf{Y}_{n,i},\bs{\theta}) - \mathbf{T}_{n} \big]\bigg\}  \\
& \quad + \sum_{m = 1}^{s - 1} \bs{\Lambda}_{n,m}^{T} \cdot \bigg\{ \sum_{j = 2}^{m} \mathbf{A}_{n,m,j}\, h C_{j-i-1,j} \big[\mathbf{J}(\vf{Y}_{n,i},\bs{\theta}) - \mathbf{T}_n\big]\bigg\}.
\end{split}
\end{equation}
\end{subequations}
It should be emphasized that \eqref{eqn:discrete_adjoint_equations} is solved from $\tf$ to $\ti$ in reverse time direction.
\item The optimality condition reads: 
\begin{eqnarray}
\label{eqn:gradient_goal_function}
0 &=& \frac{d \goalfull}{d \bs{\theta}} (\bs{\theta}) \\
 &=& \bs{\lambda}_0^T \cdot \frac{\partial \vf{y}_{\rm ini}(\bs{\theta})}{\partial \bs{\theta}} + \sum_{k=1}^N \frac{\partial q_k(\vf{y}_k,\bs{\theta})}{\partial \bs{\theta}} \notag\\
&+& \sum_{k = 0}^{N - 1} \bs{\lambda}_{k+1}^{T} \left(\mathbf{B}_{k,1}\, h \frac{\partial \vf{f}(\vf{y}_{k},\bs{\theta})}{\partial \bs{\theta}} + \displaystyle\sum_{j = 2}^{s} \mathbf{B}_{k,j}\, h  \sum_{\ell = 0}^{j -1} C_{\ell,j}\, \frac{\partial \vf{r}(\vf{Y}_{k, j-1-\ell})}{\partial \bs{\theta}}\right) \notag\\
&+& \sum_{k = 0}^{N - 1}  \sum_{m = 1}^{s - 1} \bs{\Lambda}_{k,m}^{T} \left( \mathbf{A}_{k,m,1}\, h \frac{\partial \vf{f}(y_{k},\bs{\theta})}{\partial \bs{\theta}} + \displaystyle \sum_{j = 2}^{m} \mathbf{A}_{k,m,j}\, h\sum_{\ell = 0}^{j -1} C_{\ell,j}\, \frac{\partial \vf{r}(\vf{Y}_{k, j-1-\ell})}{\partial \bs{\theta}} \right).\nonumber
\end{eqnarray}
\end{enumerate}

\subsection{Algorithmic Differentiation}
\label{sec:ad}

In the previous section, we derived the first-order-optimality conditions for discrete optimization problem in \eqref{eqn:optimal_control_problem_using_epirkw}. To derive the adjoint we use algorithmic differentiation (AD) \cite{neidinger2010introduction}, a technique that generates code to compute the sensitivities of an output with respect to an input for a given program, to compute  gradient of the Lagrangian. The site \url{www.autodiff.org} lists numerous tools that perform AD for programs written in two dozen languages. Excellent resources for the theory and implementation techniques of AD include \cite{naumann2011art,griewank2008evaluating}.

Two different modes of AD exist for first-order sensitivities - forward and reverse. The forward mode AD produces the tangent linear model of the input program, whereas the reverse mode produces the adjoint model. Mathematically, the forward mode computes Jacobian-vector products for the given program, and adjoint mode Jacobian-transposed-vector products, where the vector is some seed direction. Depending on the desire to compute a combination of the rows or columns of the Jacobian one uses the forward or adjoint mode, respectively. For a detailed discussion of cost in terms of operations and memory for each of the two modes we refer to \cite{griewank2008evaluating}. It suffices to say that for the computation of the gradient of the Lagrangian, where the number of inputs far exceeds that of the outputs, the reverse mode is more economical in terms of operation count. 

A drawback of the reverse mode is the necessity to store the entire forward program trajectory that will be required while ``reversing'' the program to compute the sensitivities. If the given program has a loop structure, then all the variables that affect the output and change from one iteration to the next have to be stored in a stack like data structure to be used during reversal of the loop in the program. As a result, storage requirements will grow with long running loop based programs. One technique to alleviate the situation is to use checkpointing \cite{griewank2008evaluating}, where periodically the stack is written to a file before being flushed for reuse. The checkpoint files are read in the reverse order while ``reversing'' the computations to calculate sensitivities. Technique are available to balance the costs of checkpointing and recomputation \cite{griewank2000algorithm}.

In this work we implemented a single-step EPIRK-W integrator in \texttt{Fortran 90} making use of the following major operations:
\begin{itemize}
	\item[-]  Linear combinations of vectors (\texttt{axpy}) are carried out using the built-in array operation syntax.
	\item[-] Matrix vector products (\texttt{gemv}) are implemented as an explicit subroutine call instead of using the Fortran intrinsic (\texttt{MATMUL}).
	\item[-] The right-hand side $\vf{f}(\vf{y})$ is an external callback function provided by the caller.
	\item[-] Computation of  $\bs{\psi}_{i,j}$ function products is rather involved and is explained in detail below.
\end{itemize}
The $\bs{\psi}_{i,j}$ products can be written as a linear combination of individual $\bs{\varphi}_k$  products, where we approximate each $\bs{\varphi}_k$ product in the Krylov-subspace  \cite{Tokman_2006_EPI} as $\bs{\varphi}_k\big(h\,\gamma\, \vf{T}_{n}\big)\vf{b} \approx \norm{\vf{b}} \vf{V}\bs{\varphi}_k(h\,\gamma\, \vf{H})\vf{e}_1$, where $\vf{V}$ is the matrix containing orthogonal basis vectors of the Krylov-subspace $\mathcal{K}_m = \;\textnormal{span}\,\{\vf{b},\vf{T}_{n}\vf{b},\vf{T}_{n}^{2}\vf{b},\hdots\}$ and $\vf{H}$ is the upper-Hessenberg matrix resulting from an Arnoldi process. Following \cite[Theorem1]{Sidje_1998}, $\bs{\varphi}_k(h\,\gamma\, \vf{H})\vf{e}_1$ products are computed by constructing an augmented matrix and exponentiating it using \texttt{Expokit} \cite{Sidje_1998}.

The adjoint of a single-step EPIRK-W integrator is obtained with the help of \texttt{Tapenade} \cite{TapenadeRef13}. We differentiate through both \texttt{axpy} and \texttt{gemv} operations, but use black-box routines (\url{http://www-sop.inria.fr/tropics/tapenade/faq.html#Libs1}) for both $\vf{f}(\vf{y})$ and $\bs{\psi}_{i,j}(h \, \gamma \, \vf{T}_{n})\vf{v}$. The adjoint of $\vf{f}(\vf{y})$ is an external callback provided by the caller. The adjoint implementation of $\bs{\psi}_{i,j}(h \, \gamma \, \vf{T}_{n})\,\vf{v}$ is based on the following result.

\begin{lemma}[Adjoint of matrix-function-times-vector operation]
	\label{lemma:adjoint_of_psi_function}

	Let $\vf{y}_{n}$ be the independent and $\vf{y}_{n+1}$ the dependent of a single-step EPIRK-W method \eqref{eqn:epirk_w_formulation_original}.
	From equations \eqref{eqn:psi_ij} and \eqref{eqn:phi_k} we have:
	\[
	\begin{split}
	\bs{\psi}_{i,j}\big(h\,\gamma\, \vf{T}_{n}\big) &= \sum_{k = 1}^{s} p_{i,j,k}\, \bs{\varphi}_k\big(h\,\gamma\, \vf{T}_{n}\big), \qquad
	\varphi_k(z) = \sum_{i=0}^\infty \frac{z^i}{(i+k)!}\,,
	\end{split}
	\] 
	where $\vf{T}_{n}$ is an arbitrary approximate Jacobian that does not depend on $\vf{y}_{n}$. The adjoint of the operation:
	\[
	\bs{\alpha} \leftarrow \bs{\psi}_{i,j}\big(h\,\gamma\, \vf{T}_{n}\big)\cdot \vf{v}
	\] 
	where  $\vf{v}$ is a vector independent of $\vf{y}_{n}$, reads:	
	\[
	{\bar{\vf{v}} \leftarrow \bs{\psi}_{i,j}\big(h\,\gamma\, \vf{T}_{n}^{T}\big)\cdot \bs{\bar{\alpha}}},
	\]
	where the adjoint variables are $\bar{\vf{v}} = \bar{\vf{y}}_{n+1}^{T} \parfrac{\vf{y}_{n+1}}{\vf{v}}$ and  $\bs{\bar{\alpha}} = \bar{\vf{y}}_{n+1}^{T} \parfrac{\vf{y}_{n+1}}{\bs{\alpha}}$.
\end{lemma}
\begin{proof}
Given an arbitrary matrix $\vf{X}$ and a vector $\vf{v}$ that is a function of the independent $\vf{y}_{n}$, the adjoint of the product $\vf{z} \leftarrow \vf{X} \vf{v}$ is the operation $\bar{\vf{v}} \leftarrow {\vf{X}^{T}} \bar{\vf{z}}$. Consequently, since $\vf{T}_n$ does not depend on $\vf{y}_{n}$, we have:
\[
\bar{\vf{v}} \leftarrow \bs{\psi}_{i,j}\big(h\,\gamma\, \vf{T}_{n}\big)^{T}\cdot\bs{\bar{\alpha}}.
\]
It remains to show that 
\begin{equation}
\bs{\psi}_{i,j}\big(h\,\gamma\, \vf{T}_{n}\big)^{T} = \bs{\psi}_{i,j}\big(h\,\gamma\, \vf{T}_{n}^{T}\big).
\end{equation}
We are given 
\begin{equation*}
\bs{\psi}_{i,j}\big(h\,\gamma\, \vf{T}_{n}\big) = \sum_{k = 1}^{s} p_{i,j,k}\, \bs{\varphi}_k\big(h\,\gamma\, \vf{T}_{n}\big).
\end{equation*}
Taking the transpose then yields  
\begin{equation}
\bs{\psi}_{i,j}\big(h\,\gamma\, \vf{T}_{n}\big)^{T} = \sum_{k = 1}^{s} p_{i,j,k}\, \bs{\varphi}_k\big(h\,\gamma\, \vf{T}_{n}\big)^{T}.
\label{eqn:psi_transposed}
\end{equation}
Since 
\begin{equation*}
\varphi_k(z) = \sum_{i=0}^\infty \frac{z^i}{(i+k)!}\,,
\end{equation*}
we have
\begin{equation}
 \bs{\varphi}_k\big(h\,\gamma\, \vf{T}_{n}\big) = \sum_{i=0}^\infty \frac{\big(h\,\gamma\, \vf{T}_{n}\big)^i}{(i+k)!} = \sum_{i=0}^\infty \frac{(h\,\gamma)^i\, (\vf{T}_{n})^i}{(i+k)!}\,.
\end{equation}
Taking the transpose then yields
\begin{equation}
\bs{\varphi}_k\big(h\,\gamma\, \vf{T}_{n}\big)^{T} = \sum_{i=0}^\infty \frac{(h\,\gamma)^i\, ((\vf{T}_{n})^i)^{T}}{(i+k)!} = \sum_{i=0}^\infty \frac{(h\,\gamma)^i\, (\vf{T}_{n}^{T})^{i}}{(i+k)!} =\bs{\varphi}_k\big(h\,\gamma\, \vf{T}_{n}^{T}\big)\,.
\label{eqn:phi_transposed}
\end{equation}
From equations \eqref{eqn:psi_transposed} and \eqref{eqn:phi_transposed}, it is not too difficult to infer 
\begin{equation}
\bs{\psi}_{i,j}\big(h\,\gamma\, \vf{T}_{n}\big)^{T} = \bs{\psi}_{i,j}\big(h\,\gamma\, \vf{T}_{n}^{T}\big),
\end{equation}
thus proving the result.
\end{proof}

\begin{remark}
If the matrix argument is the exact Jacobian $\vf{J}_{n}$, or any matrix that depends on $\vf{y}_{n}$, then the adjoint of the operation
	\[
	\bs{\alpha} \leftarrow \bs{\psi}_{i,j}\big(h\,\gamma\, \vf{J}_{n}\big)\cdot \vf{v}
	\] 
involves the additional computation of the Hessian tensor $\partial \vf{J}_{n}/\partial \vf{y}_{n}$. This is the main motivation for using a W-method formulation of EPIRK in our work: the assumption that matrices $\vf{T}_{n}$ are independent of $\vf{y}_{n}$ allows the adjoint operations to be implemented in accordance with Lemma \ref{lemma:adjoint_of_psi_function}, while the W-property allows maintaining the overall accuracy of the forward simulation.
\end{remark}

In the following section we consider two test problems where a 4D-Var problem is solved for parameter/initial state estimation.  The discrete adjoint EPIRK-W integrator is used to compute the gradients needed in optimization.

\section{Applications}
\label{sec:applications}
In this section, two different numerical examples will be given. As a first test case, we consider the Lorenz-96 model, which is a discrete nonlinear example, where the parameter is the initial condition. As a second test case, an inductor will be considered. The original equation is obtained by applying the finite element method to the magnetoquasistatic partial differential equations. Then, data assimilation is carried out for parameters characterizing the nonlinear magnetic material in the iron core. For both examples a  \texttt{Python} version of the limited memory variant of the Broyden-Fletcher-Goldfarb-Shanno (BFGS) algorithm \cite{byrd1995limited} is used for optimization. We use the third-order EPIRK-W method derived in \cite[Figure 2]{Narayanamurthi_2017_EPIRKK} in our experiments. The integrator and its adjoint is made callable from \texttt{Python} using \texttt{f2py} and checkpoints are stored in-memory.

\subsection{Lorenz-96 Model}
The model considered here was originally proposed in \cite{lorenz1996predictability} and used in a data assimilation context in \cite{singh2013practical}, where a detailed description of the setup can be found. Its main points are summarized here to enhance readability. The deterministic model is given by the set of equations:
\begin{align}
\frac{d y^j}{d t} &= -y^{j-1}(y^{j-2}-y^{j+1}-y^j) + 8, \quad j=1,\ldots,K, \nonumber \\
y^j &= y^{K+j} \quad \forall j,
\label{eqn:lorenz_96_model_equations}
\end{align}
where $K=40$ and $y=(y^1,\ldots,y^K)^T$. Let $t_0=0$, $\tf=0.3$ and the time interval be $\Delta t=0.015$ time units (time units).

We ran convergence test for the Lorenz-96 model of both the forward and the adjoint integrator with the Jacobian initialized to a random matrix at the beginning, and using a fixed random seed for the adjoint integrator. Reference solutions for both forward and adjoint problem were computed using the \texttt{MATLODE} \cite{Sandu_2015_MATLODE}  \texttt{ERK} forward and adjoint integrators, respectively, with both absolute and relative tolerance set to \num{1.0e-12}.


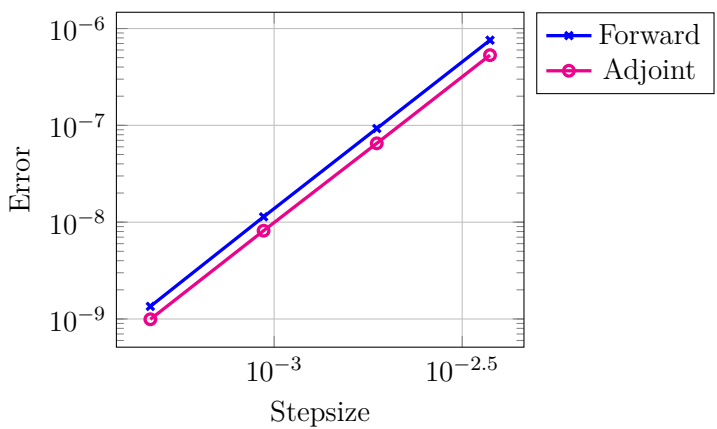
\begin{figure}[th!]
	\centering
	\pgfplotstableread{data_fixed_step_convergence.txt}{\fixedConvergence}
	\begin{tikzpicture}
	\begin{loglogaxis}[xlabel=Stepsize,ylabel=Error,ymajorgrids=true,xmajorgrids=true,legend pos = outer north east]
	\addplot[color=blue,mark=x,line width=1.25pt] table [y=y_err, x=dt_vals]{\fixedConvergence};
	\addplot[color=magenta,mark=o,line width=1.25pt] table [y=l_err, x=dt_vals]{\fixedConvergence};
	\legend{Forward,Adjoint}
	\end{loglogaxis}
	\end{tikzpicture}
	\caption{Convergence of the forward and adjoint integrator for the Lorenz-96 model \eqref{eqn:lorenz_96_model_equations}, where the error is computed with respect to reference in Euclidean norm.}
	\label{fig:fixed_step_convergence}
\end{figure}

As Figure \ref{fig:fixed_step_convergence} shows, both forward and adjoint EPIRK-W integrators achieve full third order convergence in the Euclidean norm, as desired. It should be noted, however, that either the exact Jacobian or a very good approximation of it may be needed for better stability of the integrators. In \cite{Narayanamurthi_2017_EPIRKK}, the convergence behavior of the forward EPIRK-W integrator was treated in detail. A complete treatment of the convergence behavior of the adjoint EPIRK-W integrator will be considered in a future work. We continue to describe the example in which we apply these integrators.

In this example, $\bs{\theta}$ parametrizes the initial condition as $\vf{y}_{\rm ini} = \bs{\theta}$, whereas $\vf{f}$ is a function of $\vf{y}$, solely. Following \citep{singh2013practical}, the reference $\bs{\theta}_\textsc{true}$ is obtained by integrating $y^j(-10 \Delta t)=1+0.1 \mathrm{mod}(j,5)$ in time until $t=0$ (time units). The observation operator $\mathcal{H}$ is defined as:
\begin{equation}
\mathcal{H}(\vf{y}) = (y^1,y^3,y^5,\ldots,y^{19},y^{21},y^{22},\ldots,y^{40},\sum_{j=1}^{10} y^j,\sum_{j=1}^{20} y^j,\sum_{j=21}^{40} y^j,\sum_{j=31}^{40} y^j)^T.
\end{equation}
We model the background and observation standard deviation as: 
\begin{align*}
\bs{\sigma}_\textsc{b} &= 0.03 \ \bs{\theta_\textsc{true}}, \\
\bs{\sigma}_\textsc{obs} &= \left(0.005 \left(\sum_{i=1}^{N_\textsc{obs}} \mathcal{H}(\vf{y}_{j_i}^\textsc{true}) \right)/N_\textsc{obs} \right)^{-1},
\end{align*}
where $\vf{y}_{j_i}^\textsc{true}$ refers to the solution at step $j_i = i 100$, $i=1,\ldots,N_\textsc{obs}=10$, with initial condition $\bs{\theta_\textsc{true}}$. These standard deviations are used to define the background and observation covariance matrices as
\begin{align}
\mathbf{B} &=  \alpha \mathbf{I} + (1-\alpha) \bs{\sigma} \otimes \bs{\sigma} \mathrm{e}^{-\frac{\mathbf{D}^2}{L^2}}, \nonumber \\ 
\vf{R}_{i} &= \vf{R} = \mathrm{diag}((\bs{\sigma}_\textsc{obs})^{2}),
\end{align}
with $\alpha=0.1$, $L=4$, $(\mathbf{D})_{i,j} = \min(|i-j|,K-|i-j|)$. We emphasize that $\mathbf{R}$ is modeled to be independent of time in this case. Also, in this section, $\mathbf{I}$ refers to the $K\times K$ identity matrix. The background prior $\bs{\theta}_\textsc{b}$ is generated as a pseudo-random realization according to the normal distribution of $\error_\textsc{b}$ with mean value $\bs{\theta}_\textsc{true}$ and covariance $\mathbf{B}$. Correspondingly, the measurement $\vf{z}_i$ is generated as a pseudo-random realization around the observation reference according to the normal distribution of $\error_\textsc{obs}$, with mean value $\mathcal{H}(\vf{y}_{j_i}(\bs{\theta}_\textsc{true}))$ and covariance $\mathbf{R}$. 

We perform data assimilation by solving \eqref{eqn:optimal_control_problem_using_epirkw} using the goal function and gradient given in \eqref{eqn:4dvar_cost_function} and \eqref{eqn:4dvar_gradient}, respectively. The discrete adjoints are computed using algorithmic differentiation as outlined in Section \ref{sec:ad} and supplied to the optimization routine, with the approximation $\vf{T}_{n} = \vf{J}_{n}$ used in the forward and adjoint integrators. As can be seen from Figure \ref{fig:convergence_optimization} the algorithm converges quickly. More precisely, the cost function does not change significantly after two iterations, whereas a gradient of below $\num{1.0e-4}$ is obtained within 8 iterations, for a step size of $\Delta t= 0.0003$ time units. Figure \ref{fig:convergence_optimization_parameters} depicts the difference of the initial condition at the current iteration and the true initial condition in the Euclidean norm. 
 
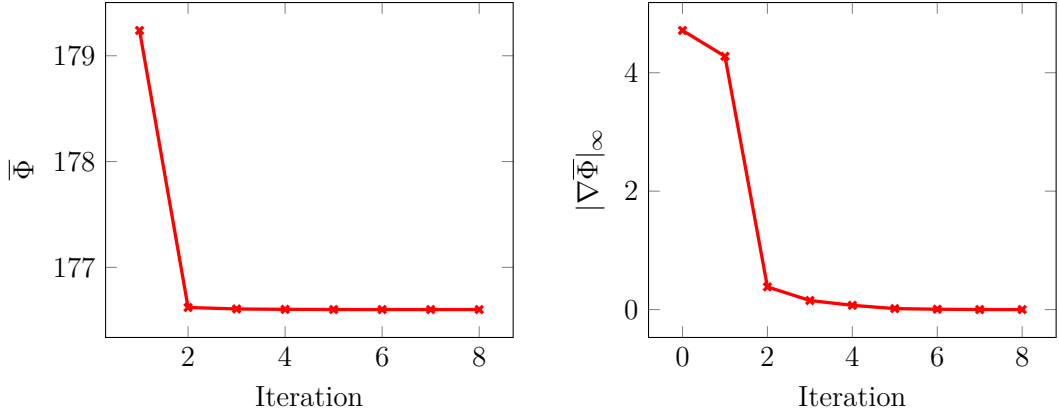
\begin{figure}
\begin{minipage}[t!]{0.49\textwidth} 
\centering
\pgfplotstableread{data_optimization_Lorenz_cost.txt}{\LorenzCost}
\begin{tikzpicture}
\begin{axis}[xlabel=Iteration,
ylabel=$\goal$]
\addplot[color=red,mark=x,restrict y to domain=170:180,line width=1.25pt] table {\LorenzCost};
\end{axis}
\end{tikzpicture}
\end{minipage}
\begin{minipage}[t!]{0.49\textwidth} 
\centering
\pgfplotstableread{data_optimization_Lorenz_gradient.txt}{\LorenzGradient}
\begin{tikzpicture}
\begin{axis}[
xlabel=Iteration,
ylabel=$|\nabla \goal|_\infty$]
\addplot[color=red,mark=x,line width=1.25pt] table {\LorenzGradient};
\end{axis}
\end{tikzpicture}
\end{minipage}
\caption{Cost function and gradient evolving during iterations of optimization algorithm for a time step size of $\Delta t = 0.0003$ time units.}
\label{fig:convergence_optimization}
\end{figure} 

\begin{figure}
\centering
\pgfplotstableread{data_optimization_Lorenz_parameters.txt}{\LorenzParameter}
\begin{tikzpicture}
\begin{axis}[xlabel=Iteration,
ylabel=$|\theta - \theta_\textsc{true}|$]
\addplot[color=red,mark=x,line width=1.25pt] table {\LorenzParameter};
\end{axis}
\end{tikzpicture}
\caption{Difference of initial condition and true initial condition in the Euclidean norm.}
\label{fig:convergence_optimization_parameters}
\end{figure}
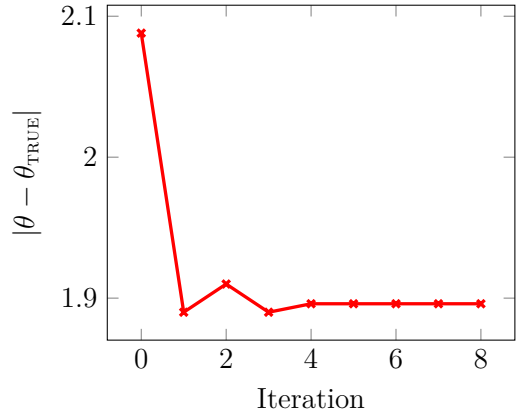

\subsection{Parameter Estimation in Magnetic Field Problems}
Before describing the data assimilation problem for the inductor, we give some background of the underlying partial differential equations. The magnetoquasistatic problem in a bounded computational domain $\Omega$ reads as
\begin{subequations}
\begin{align}
\curl \vec{E} &= - \parfrac{\vec{B}}{t}, \quad &&\mathrm{in} \ \Omega, \\
\curl \vec{H} &= \vec{J} + \kappa \vec{E}, \quad &&\mathrm{in} \ \Omega, \\
\vec{B} \cdot \vec{n} &= 0, \quad &&\mathrm{on} \ \partial \Omega,
\end{align}
\label{eq:eddy_current}%
\end{subequations}
and is widely used in the simulation of electrical machines, magnets and transformers. In \eqref{eq:eddy_current}, $\vec{E}$ denotes the electric field, $\kappa$ the conductivity, $\vec{J}$ the imposed source current density and $\vec{n}$ the outer unit normal. Also, the magnetic flux density $\vec{B}$ and the magnet field strength $\vec{H}$ are related through the nonlinear material relation $\vec{H}= \nu(|\vec{B}|)\vec{B}$. Here, we restrict ourselves to parametric nonlinearities of the type
\begin{equation}
\nu(|\vec{B}|) = \frac{1}{2 |\vec{B}|}\left(\tanh \left(\frac{|\vec{B}|}{\theta_1}\right) + \tanh \left(\frac{|\vec{B}|}{\theta_2}\right)^{30}\right) \left(\theta_3 + \theta_4 |\vec{B}|\right),
\label{eqn:nonlinearity}
\end{equation}
see \cite{rosseel2010nonlinear}, where $\nu(0)$ is obtained by taking the limit $|\vec{B}|\rightarrow 0$ in the above equation. The parameters $\bs{\theta}$ of the nonlinearity are often deduced from measurements. Equation \eqref{eq:eddy_current} is discretized in space by applying the finite element method. To this end, as a first step, from \eqref{eq:eddy_current} we derive the second order problem
\begin{subequations}
\begin{align}
\kappa \frac{\partial \vec{A}}{\partial t} + \curl \left( \nu(\bs{\theta},|\curl \vec{A}|) \curl \vec{A} \right) &= \vec{J}, \quad &&\mathrm{in} \ \Omega, \\
\vec{A} \times \vec{n} &= 0, \quad &&\mathrm{on} \ \partial \Omega,
\end{align}
\label{eq:eddy_current_second_order}%
\end{subequations}
with the magnetic vector potential $\vec{A}$ such that $\vec{B}= \curl \vec{A}$ and the initial condition $\vec{A}(0)= \vec{A}_0$ in $\Omega$, see, e.g., \cite{biro1999edge}. We further simplify the problem by considering a two-dimensional domain. By inserting $\vec{A}=(0,0,u)^T$ in \eqref{eq:eddy_current_second_order}, where the third coordinate represents the direction perpendicular to the two-dimensional domain, we obtain
\begin{subequations}
\begin{align}
\kappa \frac{\partial u}{\partial t} - \nabla \left( \nu(\bs{\theta},|\nabla u|) \nabla u \right) &= f, \quad &&\mathrm{in} \ \Omega, \\
u &= 0, \quad &&\mathrm{on} \ \partial \Omega.
\end{align}
\label{eq:eddy_current_second_order_two_dimensional}%
\end{subequations}
Difficulties arise in the solution of problem \eqref{eq:eddy_current_second_order_two_dimensional}, as $\kappa$ vanishes in non-conducting sub-regions of $\Omega$. A remedy consists in using $\kappa_\text{reg} = \max(\kappa,\kappa_\text{air})$, with artificial conductivity $\kappa_\text{air} > 0$, in \eqref{eq:eddy_current_second_order} instead. The error of this approximation as a function of $\kappa_\text{air}$ is well understood, see, e.g., \cite{bachinger2005numerical}. We assume that $\kappa$ is replaced with $\kappa_\text{reg}$ from now on and omit the subscript for simplicity. 

To state the weak formulation of \eqref{eq:eddy_current_second_order_two_dimensional} we introduce the usual Sobolev space
\begin{equation}
V = \{u \in L^2(\Omega) \ | \ \nabla u \in (L^2(\Omega))^2 \ \text{and} \ u|_{\partial \Omega} = 0\}. 
\end{equation}
Then, we seek $u: [t_0,\tf] \rightarrow V$ such that 
\begin{equation}
\int_{\Omega} \kappa \frac{\partial u}{\partial t} v \ \mathrm{d}x + \int_{\Omega} \nu(\bs{\theta},|\nabla u|) \nabla u \cdot \nabla v \ \mathrm{d}x = \int_{\Omega} f v \ \mathrm{d}x, \ \forall  v \in V.
\end{equation}

The finite element space $V_h \subset V$ consists of piecewise linear polynomials on a quasi-uniform triangulation of $\Omega$. This results in an ordinary differential equation
\begin{equation}
\mathbf{M} \vf{y}_\text{FE}' + \mathbf{K}(\vf{y},\bs{\theta}) \vf{y}_\text{FE} = \vf{f}_\text{FE},
\label{eq:eddy_current_discrete}
\end{equation}
where $\mathbf{M},\mathbf{K},\vf{y}_\text{FE},\vf{f}_\text{FE}$ refer to the mass and stiffness matrix and to the vector of degrees of freedom and the source vector, respectively. Problem \eqref{eq:eddy_current_discrete} can be recast as
\begin{equation}
\vf{y}_\text{FE}' = \vf{f}(\vf{y}_\text{FE},\bs{\theta}), \qquad \vf{y}_\text{FE} = \vf{y}_{\rm ini},
\label{eq:eddy_current_discrete_modified}
\end{equation}
where $\vf{f}(\vf{y}_\text{FE},\bs{\theta})= \mathbf{M}^{-1}(\vf{f}_\text{FE} - \mathbf{K}(\vf{y}_\text{FE},\bs{\theta})\vf{y}_\text{FE})$. From \eqref{eq:eddy_current_discrete_modified} it is clearly visible that the parameter dependence is bound to $\vf{f}$ solely.


The example considered here is an inductor, introduced in \cite{meeker}, which is depicted in Figure \ref{fig:inductor}. The observed (scalar) quantity is the flux linkage defined as 
\begin{equation}
\mathcal{H} = \frac{N_\text{turns}}{|\Omega_\text{coil}|} \int_{\Omega_\text{coil}} \frac{\partial u}{\partial t} \ \mathrm{d}x, 
\end{equation}
where $N_\text{turns}$, $\Omega_\text{coil}$ and $|\Omega_\text{coil}|$ refer to the number of turns in the coil, the coil domain (red in Figure \ref{fig:inductor}) and its area, respectively.

\begin{figure}[t!]	
\begin{minipage}[h!]{0.55\textwidth}
\centering
\begin{tikzpicture}
\node at (0.99,-0.03) {\includegraphics[scale=0.15]{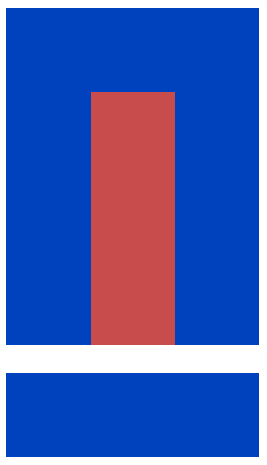}};
\node at (3.15,0) {\includegraphics[scale=0.15]{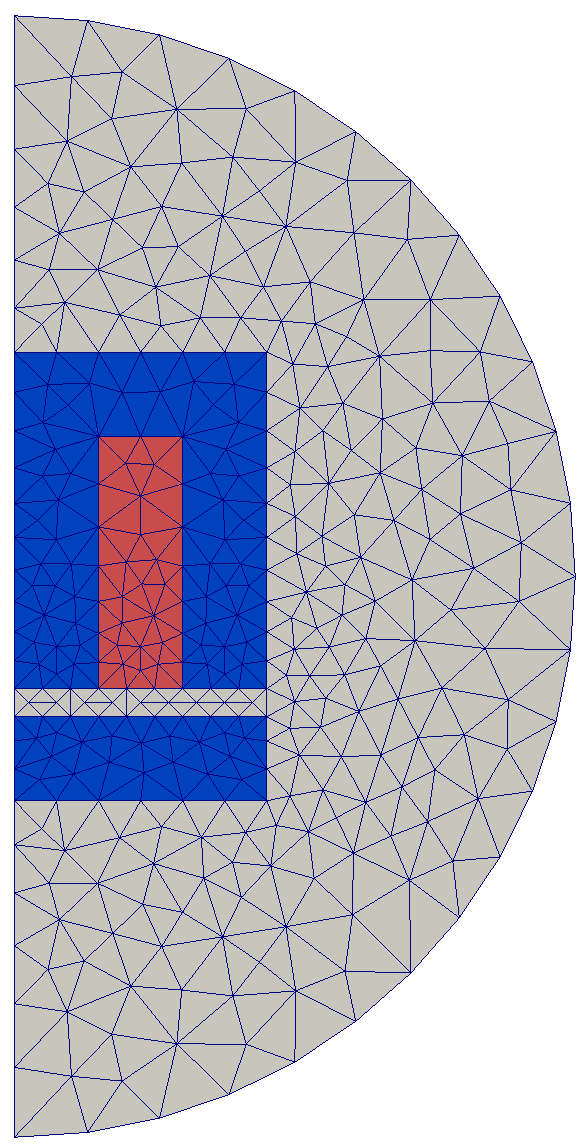}};
\end{tikzpicture}
\end{minipage}
\begin{minipage}[h!]{0.44\textwidth}
\begin{tikzpicture}
\node at (0,0.5) {\includegraphics[scale=0.15]{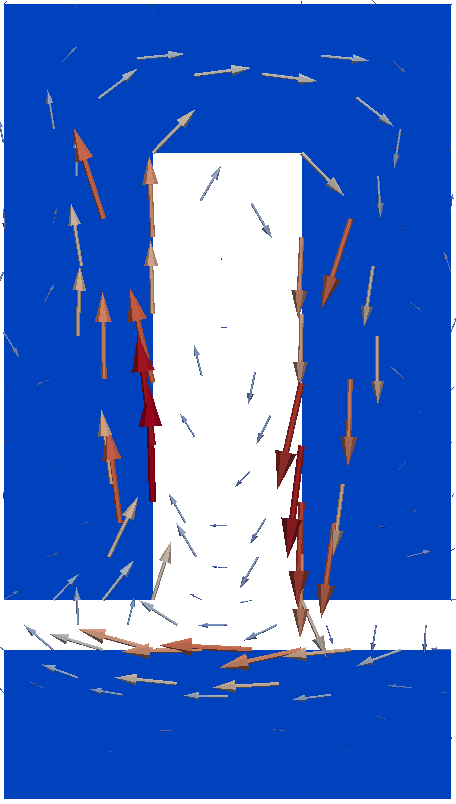}};
\node at (2.3,2.4) {\small $|\vec{B}|$ (T)};
\node at (2.4,0.3) {\includegraphics[scale=0.2]{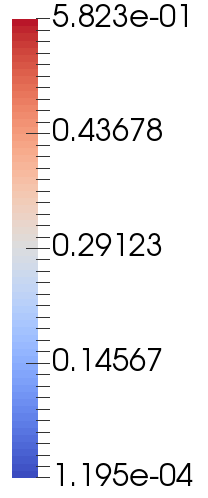}};
\end{tikzpicture}
\end{minipage}
\caption{Left: inductor configuration with airgap, coil (red) and iron core (blue). Due to symmetry only the right part is used for the simulation. The corresponding computational domain is depicted (air in gray) with the triangular mesh. Right: magnetic flux density in the iron core.}
\label{fig:inductor}
\end{figure}

The parameter estimation problem is set up as follows. We assume that the true parameters of the nonlinearity \eqref{eqn:nonlinearity} are given by:
\[
\bs{\theta_\textsc{true}} = \big[\num{2.88e+3},\num{5.99e+7}, \num{4.35e+4},1.89\big].
\]
Here, the background and measurement covariance are modeled as:
\begin{align*}
\mathbf{B} &= \mathrm{diag}(0.09 \ \theta_i^2), \\
\mathbf{R}_{i} &=  (\sigma_{\textsc{obs}})^2,
\end{align*}
i.e., both the parameters and the measurements are assumed to be uncorrelated. We consider a sinusoidal current excitation $I(t) = I_0 \sin(4 \pi t)$ on the time interval $[0,0.5]$ s. Measurements are taken at time steps $j_i = 10 i$, for $i=1,\ldots,5=N_\textsc{obs}$. The measurement standard deviation is modeled to be $0.1$ percent of the effective measurement value:
\[
\sigma_{\textsc{obs}} = 0.001\, \sqrt{\sum_{i=1}^{N_\textsc{obs}} \mathcal{H}(\vf{y}_{j_i}(\bs{\theta}_\textsc{true}))^2}.
\]
Then, the background prior $\bs{\theta}_\textsc{b}$ is generated as a pseudo-random realization according to a normal distribution with mean $\bs{\theta_\textsc{true}}$ and covariance $\mathbf{B}$. Accordingly, measurements $z_i$ are generated as pseudo-random realizations, according to a normal distribution with covariance $\vf{R}$ around the \emph{true} trajectory. 

With the goal function and gradient as in equations \eqref{eqn:4dvar_cost_function} and \eqref{eqn:4dvar_gradient}, we apply the EPRIK-W method to solve \eqref{eq:eddy_current_discrete_modified}. This problem is very stiff and it warrants the use of the exact Jacobian, i.e. $\vf{T}_{n} = \vf{J}_{n}$, to enhance stability. We allow for a 20 percent variation of the coefficients during optimization. For the simulation a step size of $\Delta t = 0.01$ s is employed. Using symmetry only the right half of the configuration is used in the computations. A current of $I_0 = 150$ A per turn, with $N_\text{turns}=66$ turns in total, is imposed to the coil. The associated flux distribution in the core is depicted in Figure \ref{fig:inductor} on the right. Meshing and finite element analysis are carried out with Gmsh \cite{geuzaine2009gmsh} and FEniCS \cite{AlnaesBlechta2015a,LoggMardalEtAl2012a}, respectively. To reduce the stiffness of the problem we set $\kappa_\text{air} = \num{1.0e+6} \ \mathrm{S}\mathrm{m}^{-1}$, which has a negligible effect on the flux linkage. 

Figure \ref{fig:convergence_optimization_inductor} shows the cost function and the gradient over the iteration steps of the optimization. Again, we observe a fast convergence of the cost function. In Figure \ref{fig:parameter_convergence_inductor} we plot the difference of parameter $\theta_4$ and its true value, during optimization. The error in $\theta_4$ decays quickly, while the other parameters remain unchanged. This can be attributed to the fact that the derivative of the objective function with respect to $\theta_4$ is orders of magnitudes larger than the derivatives with respect to $\theta_1, \theta_2$ and $\theta_3$. The observation is further supported by the fact that $\theta_4$ models the saturated range of the nonlinearity \cite{rosseel2010nonlinear}, which contributes significantly to the shape of the waveform depicted in Figure \ref{fig:assimilation_inductor}. As this shape is modified during data assimilation to fit the measurements, $\theta_4$ plays a key role here. Figure \ref{fig:assimilation_inductor} depicts the flux linkage over the time interval for the prior and the parameters after assimilation together with the measurements. The trajectory after data assimilation is in better accordance with the measurement data. Finally, Table \ref{tab:4dvar_inductor} summarizes the measurement values as well as the numerical predictions before and after data assimilation. 

\begin{figure}[t!]
\begin{minipage}[t!]{0.49\textwidth} 
\centering
\pgfplotstableread{data_optimization_inductor_cost.txt}{\inductorCost}
\begin{tikzpicture}
\begin{axis}[xlabel=Iteration,
ylabel=$\goal$]
\addplot[color=red,mark=x,line width=1.25pt] table {\inductorCost};
\end{axis}
\end{tikzpicture}
\end{minipage}
\begin{minipage}[t!]{0.49\textwidth} 
\centering
\pgfplotstableread{data_optimization_inductor_gradient.txt}{\inductorGradient}
\begin{tikzpicture}
\begin{axis}[
xlabel=Iteration,
ylabel=$|\nabla \goal|_\infty$]
\addplot[color=red,mark=x,line width=1.25pt] table {\inductorGradient};
\end{axis}
\end{tikzpicture}
\end{minipage}
\caption{Goal function and projected gradient evolving during iterations of optimization algorithm for a time step size of $\Delta t = 0.01$.}
\label{fig:convergence_optimization_inductor}
\end{figure}
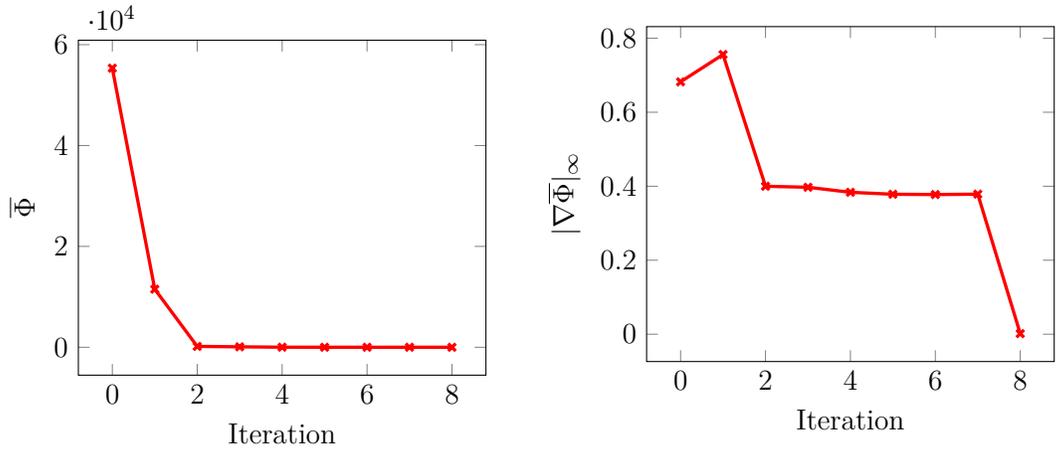
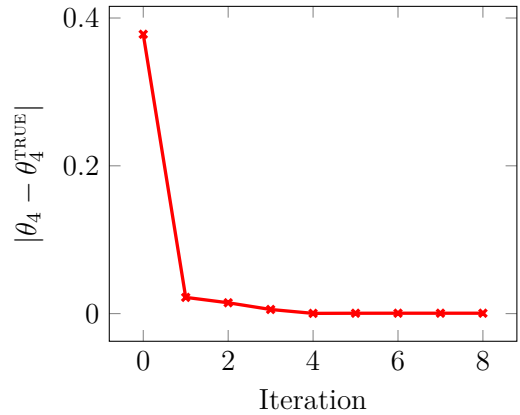
\begin{figure}[t!]
	\centering
	\pgfplotstableread{data_optimization_inductor_parameter_convergence.txt}{\inductorMeasErr}
	\begin{tikzpicture}
	\begin{axis}[xlabel=Iteration,ylabel=$\abs{\vf{\theta}_4 - \vf{\theta}_4^\textsc{true}}$]
	\addplot[color=red,mark=x,line width=1.25pt] table [y=Errortheta4, x=Iteration]{\inductorMeasErr};
	\end{axis}
	\end{tikzpicture}
	\caption{Absolute error in predicted $\theta_4$ from true $\theta_4^{\textsc{true}}$. Error for other components remained constant and are as follows $\abs{\vf{\theta}_1 - \vf{\theta}_1^\textsc{true}} = 6.496$, $\abs{\vf{\theta}_2 - \vf{\theta}_2^\textsc{true}} = 1443.266$, $\abs{\vf{\theta}_3 - \vf{\theta}_3^\textsc{true}} = 65.979$, and in a relative sense \num{2e-3}, \num{2.4e-5}, \num{1.0e-3} , respectively.}
	\label{fig:parameter_convergence_inductor}
\end{figure}
\begin{figure}[t!]
\centering
\pgfplotstableread{data_optimization_inductor_meas.txt}{\inductorMeas}
\pgfplotstableread{data_optimization_inductor_trajectory.txt}{\inductorBackground}
\pgfplotstableread{data_optimization_inductor_trajectory_2.txt}{\inductorOptimal}
\begin{tikzpicture}
\begin{axis}[xlabel=$t$,
ylabel=flux linkage, legend pos = outer north east]
\addplot[color=black,mark=square, only marks, y = 2nd col,mark size=1.5] table {\inductorMeas};
\addplot[color=blue,dashed,line width=1.25pt] table {\inductorBackground};
\addplot[color=magenta,line width=1.25pt] table {\inductorOptimal};
\legend{Measurements,Background,Optimum}
\end{axis}
\end{tikzpicture}
\caption{Measurements and simulation, with parameters before and after parameter estimation, of flux linkage over time}
\label{fig:assimilation_inductor}
\end{figure}
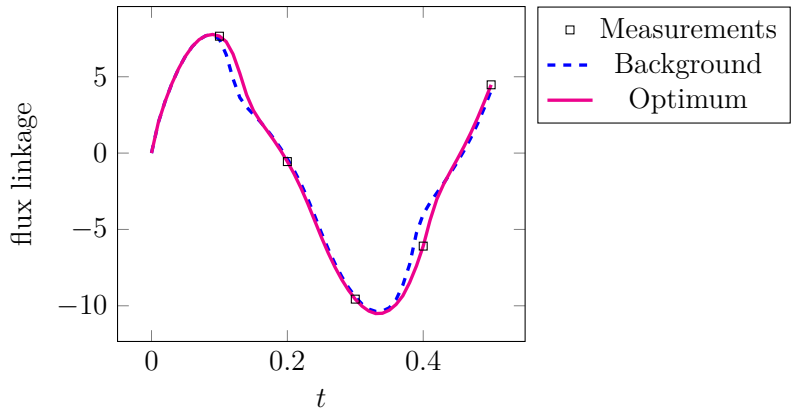  

\begin{table}[t!]
\centering
\begin{tabular}{|c|c|c|c|}
\hline
 & Measurement & Background & Optimum \\
 \hline
$z_1$ & 7.656 & 7.450 & 7.660\\
\hline
$z_2$ & -0.556 & -0.377 & -0.554 \\
\hline
$z_3$ & -9.565 & -9.414 & -9.571 \\
\hline
$z_4$ & -6.092 & -4.004 & -6.092\\
\hline
$z_5$ & 4.472 & 4.098 & 4.470 \\
\hline 
\end{tabular}
\caption{Measurements compared to background simulation and simulation after data assimilation in $\mathrm{Wb}\mathrm{s}^{-1}$.}
\label{tab:4dvar_inductor}
\end{table}


\section{Conclusions}
\label{sec:conclusions}
This work derives the discrete adjoint formulae for a general class of exponential integrators, EPIRK-W. The choice of methods with the W-property allows to use arbitrary approximations of the Jacobian as arguments of the matrix functions while maintaining the overall accuracy of the forward integration. The use of matrices that do not depend on the model state avoids the complex calculation of Hessians (derivatives of the Jacobian with respect to state variables) in the discrete adjoint formulae. The simplified discrete adjoint can computed via algorithmic differentiation and then supplied to optimization routines for solving the inverse problem at hand. We were able to empirically verify that the convergence order of the discrete adjoint of the EPIRK-W method matches the order of convergence of the forward EPIRK-W integrator. The methodology was applied to estimate the initial condition of the Lorenz-96 model from synthetic measurement data. The methodology was also applied to estimate magnetic material parameters for computational magnetics problems, where a system of stiff ordinary differential equations is obtained by a method of lines approach. Exponential integrators are a promising alternative for this class of problems and this topic merits further investigation. 

\section*{Acknowledgements}
\label{sec:ack}
The work of M. Narayanamurthi and A. Sandu was supported in part by awards NSF DMS--1419003, NSF CCF--1613905, AFOSR DDDAS 15RT1037, and by the Computational Science Laboratory at Virginia Tech. Part of the work was performed during {U.~R\"omer's} visit at Virginia Tech.

\bibliographystyle{gOMS}
\bibliography{adjoints_exp,ode_exponential,ode_general,ode_krylov,sandu}

\end{document}

%% file: logo.tex
\thispagestyle{empty}
\setcounter{page}{0}

\makeatletter
\def\Year#1{%
  \def\yy@##1##2##3##4;{##3##4}%
  \expandafter\yy@#1;
}
\makeatother

\begin{Huge}
\begin{center}
Computational Science Laboratory Technical Report CSL-TR-\Year{\the\year}-{\tt 4} \\
\today
\end{center}
\end{Huge}
\vfil
\begin{huge}
\begin{center}
Ulrich R\"omer, Mahesh Narayanamurthi, and Adrian Sandu
\end{center}
\end{huge}

\vfil
\begin{huge}
\begin{it}
\begin{center}
``{\tt Solving Parameter Estimation Problems with Discrete Adjoint Exponential Integrators}''
\end{center}
\end{it}
\end{huge}
\vfil

\begin{large}
\begin{center}
Computational Science Laboratory \\
Computer Science Department \\
Virginia Polytechnic Institute and State University \\
Blacksburg, VA 24060 \\
Phone: (540)-231-2193 \\
Fax: (540)-231-6075 \\ 
Email: \url{maheshnm@vt.edu} \\
Web: \url{http://csl.cs.vt.edu}
\end{center}
\end{large}

\vspace*{1cm}

\begin{tabular}{ccc}
\includegraphics[width=2.5in]{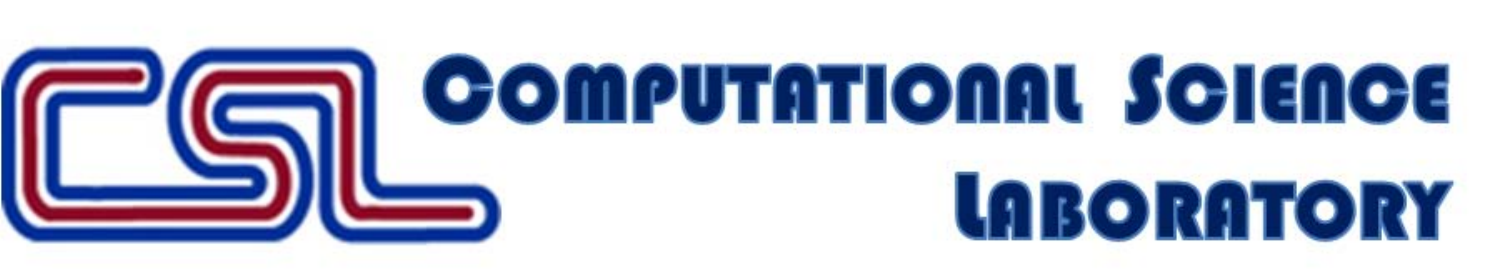}
&\hspace{2.5in}&
\includegraphics[width=2.5in]{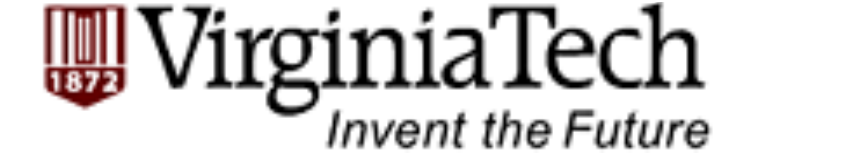} \\
{\bf\large \textit{Compute the Future}} &&\\
\end{tabular}

\newpage